\documentclass[]{interact}

\usepackage{epstopdf}% To incorporate .eps illustrations using PDFLaTeX, etc.
\usepackage{subfigure,hyperref}% Support for small, `sub' figures and tables

\usepackage[numbers,sort&compress]{natbib}% Citation support using natbib.sty
\bibpunct[, ]{[}{]}{,}{n}{,}{,}% Citation support using natbib.sty
\makeatletter% @ now becomes a letter
\def\NAT@def@citea{\def\@citea{\NAT@separator}}% Macro to suppress spaces between citations using natbib.sty
\makeatother% @ becomes a symbol again

\theoremstyle{plain}% Theorem-like structures provided by amsthm.sty
\newtheorem{theorem}{Theorem}[section]
\newtheorem{lemma}[theorem]{Lemma}
\newtheorem{corollary}[theorem]{Corollary}
\newtheorem{proposition}[theorem]{Proposition}

\theoremstyle{definition}

\newtheorem{example}[theorem]{Example}

\theoremstyle{remark}
\newtheorem{remark}{Remark}

\begin{document}
\fontsize{11pt}{15pt}
	\title{Hidden convexity of quadratic systems and its application to quadratic programming}
	
	\author{
		\name{Nguyen Quang Huy\textsuperscript{a}, Nguyen Huy Hung\textsuperscript{a}, Tran Van Nghi\textsuperscript{a},\\ Hoang Ngoc Tuan\textsuperscript{a}, Nguyen Van Tuyen\textsuperscript{a}\thanks{Contact: Tran Van Nghi, email: tranvannghi@hpu2.edu.vn}}
		\affil{\textsuperscript{a}Department of Mathematics, Hanoi Pedagogical University 2,\\
		\;  Xuan Hoa, Phu Tho, Vietnam}
	}
	
	\maketitle
	
	\begin{abstract}
		In this paper, we present sufficient conditions ensuring that the sum of the image of quadratic functions and the nonnegative orthant
		is convex.
		The hidden convexity of the trust-region problem with linear inequality constraints
		is established under a newly proposed assumption, which is compared with the previous one in
		[{\it Math. Program. 147, 171--206, 2014}].
		We also provide a complete proof of the hidden convexity of a system of two quadratic functions
		in [{\it J. Glob. Optim. 56, 1045--1072, 2013}].
		Furthermore, necessary and sufficient conditions for the S-lemma concerning systems of quadratic inequalities
		are investigated.
		Finally, we derive necessary and sufficient global optimality conditions and strong duality results
		for quadratic programming.
	\end{abstract}
	
	
	
	\section{Introduction}
	Quadratic programming plays a fundamental role in a wide range of applications
	and frequently appears as a subproblem in numerous optimization algorithms.
	Many real-world problems, such as those arising in planning and scheduling,
	economies of scale, engineering design, and control, can be naturally formulated
	within the framework of quadratic programming
	(see \cite{FloudasVisweswaran1995,LTY05book,Tuy2016}).
	
	The classical S-lemma constitutes a powerful analytical tool, particularly in control theory
	and, more specifically, in the stability analysis of nonlinear systems
	(see \cite{PolikTerlaky2007}).
	Over its more than eighty-year development, the S-lemma has been established and generalized
	through a variety of approaches.
	It has also found significant applications in quadratic and semidefinite optimization.
	Within the context of nonconvex optimization, the S-lemma often uncovers hidden convexity
	in otherwise nonconvex problems.
	When applied to quadratic optimization, it yields strong duality conditions
	for problems involving quadratic constraints
	(see \cite{Tuy2016,BeckEldar2006,TuyTuan2013,JeyakumarHuyLi2009,JeyakumarLi2014,PolikTerlaky2007}).
	
	Let $f$ and $g_i$, $i=1,\ldots,m$, be quadratic functions on $\mathbb{R}^n$.
	It is well known that the generalized S-lemma for quadratic functions
	$f, g_1,\ldots,g_m$ with $m\geq 2$ holds if Slater's condition is satisfied
	and if the set ${\rm U}(f,g_1,\ldots,g_m)$ is convex, where
	$${\rm U}(f,g_1,\ldots,g_m):=
	\{(f(x), g_1(x), \ldots, g_m(x)) : x\in \mathbb{R}^n\} + \mathbb{R}^{m+1}_+$$
	(see, for instance, \cite{Fradkov1973,LaraUrruty2022}).
	This approach enables us to derive the generalized S-lemma and many of its extensions
	in a stronger form and in a unified and more direct manner than existing results.
	Consequently, the study of the convexity of the set ${\rm U}(f,g_1,\ldots,g_m)$
	plays a central role in quadratic programming.
	In particular, it provides a more effective framework for investigating
	global optimality conditions and strong duality in nonconvex quadratic optimization.

	Furthermore, the convexity of ${\rm U}(f,g_1,\ldots,g_m)$ is equivalent to the property that the function
	$$(x, \gamma_0, \gamma_1,\ldots, \gamma_m)\mapsto \gamma_0 f(x)+\sum_{i=1}^{m}\gamma_i g_i(x)$$
	is convexlike with respect to $x\in \mathbb{R}^n$ for every
	$(\gamma_0, \gamma_1,\ldots,\gamma_m)\in \mathbb{R}^{m+1}$
	(see \cite{Fan1953,Roubi2022}).
	As a consequence, the minimax inequality and the generalized Yuan's lemma hold
	(see \cite{Roubi2022}). 
	
	It is clear that the hidden convexity assumption for quadratic systems is strictly weaker
	than the convexity of the image itself, which has been investigated in
	\cite{Dines1941,Polyak1998,BazanOpazo2016,QuangChuSheu2022,BazanOpazo2021}.

	For $m=1$, Polyak \cite{Polyak1998} proved the convexity of ${\rm U}(f,g_1)$ under the assumption that $g_1$
	is a strictly convex quadratic function.
	Subsequently, Tuy and Tuan \cite{TuyTuan2013,Tuy2016} showed that ${\rm U}(f,g_1)$ is convex without imposing any additional conditions.
	More recently, this important result was also established in \cite{BazanOpazo2016}.
	
	For the case $m\geq 2$, the authors of \cite[Example 2.1]{JeyakumarLi2014} constructed an example in which the set
	${\rm U}(f,g_1,g_2)$ is not convex.
	This demonstrates that hidden convexity does not generally hold for systems of quadratic inequalities.
	A natural question therefore arises:
	
	``{\it Under what assumptions is the set ${\rm U}(f,g_1,\ldots,g_m)$ convex}?''
	
	Under the assumptions that $f$ is an arbitrary quadratic function, $g_1$ is a strictly convex quadratic function,
	all remaining functions $g_2,\ldots,g_m$ are affine, and a suitable dimension condition
	(see \cite{JeyakumarLi2014}) is satisfied, the set ${\rm U}(f,g_1,\ldots,g_m)$ is convex.
	In \cite{BomzeJeyakumarLi2018}, the authors investigated the CDT (Celis--Dennis--Tapia, or two-ball trust-region) problem
	and proposed an extended dimension condition ensuring the convexity of ${\rm U}(f,g_1,\ldots,g_m)$.
	Another convexity result, under the assumption that the associated matrices are $Z$-matrices,
	can be found in \cite{JeyakumarLeeLi2009}.
	However, the assumptions employed to guarantee hidden convexity in these works are rather restrictive.
	Therefore, it is both necessary and important to establish weaker and easily verifiable assumptions
	for the study of quadratic programming.

	In this paper, we provide a partial answer to the above open question.
	Our main contributions to quadratic programming under the newly proposed assumptions
	can be summarized as follows:
	
	$\bullet$ Sufficient conditions guaranteeing the convexity of
	${\rm U}(f,g_1,g_2,\ldots,g_m)$ are established under the proposed assumptions;
	
	$\bullet$ The hidden convexity of the trust-region problem with linear inequality constraints
	is investigated, and the proposed assumptions are compared with existing ones in
	\cite{BomzeJeyakumarLi2018,JeyakumarLi2014,Polyak1998};
	
	$\bullet$ A complete proof of the hidden convexity of a system of two quadratic functions
	is provided;
	
	$\bullet$ Necessary and sufficient conditions for the S-lemma associated with systems of quadratic inequalities
	are studied;
	
	$\bullet$ Necessary and sufficient global optimality conditions and strong duality results
	for quadratically constrained quadratic optimization problems are derived under the proposed assumptions,
	and these results are compared with existing ones in
	\cite{BomzeJeyakumarLi2018,JeyakumarLi2014,Tuy2016,TuyTuan2013}.
	
	Our main tools rely on properties of quadratic functions and recession cones.
	Several of our assumptions are weaker than those commonly used in the existing literature.
	In addition, we present numerical examples that not only illustrate the theoretical results
	but also demonstrate situations in which the existing results are not applicable.
	
	The remainder of the paper is organized as follows.
	Section~2 introduces the problem formulation and preliminary results.
	In Section~3, we establish hidden convexity results for quadratic systems
	and investigate the hidden convexity of the trust-region problem with linear inequality constraints,
	including a comparison with previous assumptions.
	Section~4 is devoted to a complete proof of the hidden convexity of a system of two quadratic functions.
	Finally, in the last section, we study necessary and sufficient conditions for the S-lemma
	associated with systems of quadratic inequalities, as well as strong duality results.

	\section{ Notation and preliminaries} 
	
	For any positive integer $n$, $\mathbb{R}^n$ denotes the $n$-dimensional real Euclidean space
	equipped with the scalar product $\langle \cdot,\cdot \rangle$ and the induced norm $\|\cdot\|$.
	Let $\mathbb{R}^{n\times n}_{S}$ denote the space of real symmetric $(n\times n)$ matrices
	endowed with the matrix norm induced by the vector norm in $\mathbb{R}^n$.
	Let $\mathbb{R}^{n \times n}_{S^{+}}$
	denote the set of positive semidefinite real symmetric $(n\times n)$ matrices.
	
	The scalar product of vectors $x$ and $y$ and the Euclidean norm of a vector $x$
	in a finite-dimensional Euclidean space are denoted, respectively, by $x^T y$
	(or $\langle x,y\rangle$) and $\|x\|$, where the superscript $^T$ denotes transposition.
	Vectors in Euclidean spaces are interpreted as columns of real numbers.
	The notation $x \ge y$ (resp., $x > y$) means that every component of $x$
	is greater than or equal to (resp., greater than) the corresponding component of $y$.

	Let $S\subset \mathbb{R}^n$ be a nonempty closed convex set.
	The recession cone of $S$ is defined in \cite[p.~61]{Rockafellar1970} by
	$$
	S^\infty:=\{v\in \mathbb{R}^n : x+tv\in S \ \forall x\in S,\ \forall t\geq 0\}.
	$$
	It follows directly from the definition that $S+S^\infty\subset S$.
	Clearly, $S\subset S+S^\infty$ since $0\in S^\infty$.
	Therefore, $S=S+S^\infty$.
	According to \cite[Theorem~8.3]{Rockafellar1970}, one also has
	$$
	S^\infty:=\{v\in \mathbb{R}^n : \exists x\in S \text{ such that } x+tv\in S \ \forall t\geq 0\}.
	$$
	The recession cone of a set defined by convex quadratic inequalities can be characterized as follows.
	
	\begin{lemma}\label{lemmanlx} {\rm (see \cite{BK02})}
		Let
		\begin{equation}\label{rb}
			{\cal F}:=\big\{x\in \mathbb{R}^n:
			\dfrac{1}{2}x^TQ_ix+q_i^Tx+c_i \leq 0,\ i=1,\ldots,m\big\},
		\end{equation}
		where
		$Q_i \in\mathbb{R}^{n \times n}_{S^{+}}$, $q_i\in \mathbb{R}^n$, and $c_i\in \mathbb{R}$
		for every $i=1,\ldots,m$.
		If ${\cal F}$ is nonempty, then
		\begin{equation}\label{c1a}
			{\cal F}^\infty
			=\big\{v\in \mathbb{R}^n : Q_iv=0,\ q_i^Tv \leq 0,\ i=1,\ldots,m\big\}.
		\end{equation}
	\end{lemma}
	
	In this paper, we fix
	$$
	f(x):=x^TAx+a^Tx+\alpha
	$$
	and
	$$
	g_i(x):=x^TB_ix+b_i^Tx+\beta_i,\quad i=1,\ldots,m,
	$$
	where
	$A\in \mathbb{R}^{n \times n}$,
	$B_i \in\mathbb{R}^{n \times n}_{S^{+}}$,
	$a,b_i\in \mathbb{R}^n$, and $\alpha, \beta_i\in \mathbb{R}$
	for every $i=1,\ldots,m$.
	
	We denote
	$$
	[h\leq 0]:=\{x\in \mathbb{R}^n : h(x)\leq 0\}.
	$$
	
	Let $h:\mathbb{R}^n \rightarrow \mathbb{R}\cup\{-\infty,+\infty\}$.
	The conjugate function of $h$,
	$$
	h^*:\mathbb{R}^n \rightarrow \mathbb{R}\cup\{+\infty\},
	$$
	is defined by
	$$
	h^*(v):=\sup\{v(x)-h(x): x\in {\rm dom}\, h\},
	$$
	where
	$$
	{\rm dom}\, h:=\{x\in \mathbb{R}^n : h(x)<+\infty\}
	$$
	is the effective domain of $h$ and $v(x)=v^Tx$.
	The function $h$ is said to be proper if it never takes the value $-\infty$
	and ${\rm dom}\, h\neq \emptyset$.
	The epigraph of $h$ is defined by
	$$
	{\rm epi}\, h:=\{(x,r)\in \mathbb{R}^n\times \mathbb{R} :
	x\in {\rm dom}\, h,\ h(x)\leq r\}.
	$$
	
	If $h$ is a continuously differentiable convex function on $\mathbb{R}^n$
	and $\bigcup_{\lambda\geq 0} {\rm epi}\, (\lambda h)^*$ is closed, then for each
	$x\in [h\leq 0]$ (see \cite{JeyakumarLee2008}), one has
	\begin{equation}\label{h4}
		N_{[h\leq 0]}(x)
		=\bigcup_{\lambda\geq 0,\ \lambda h(x)=0} \{\lambda \nabla h(x)\}.
	\end{equation}
	
	A matrix $M\in \mathbb{R}^{n \times n}_{S}$ is said to be positive semidefinite
	(positive definite, respectively), denoted by $M\succeq 0$ ($M\succ 0$, resp.),
	if $x^TMx\geq 0$ ($x^TMx>0$, resp.) for all $x\in \mathbb{R}^n$.

	\section{Main results}

	In the following main results, we present sufficient conditions for the convexity of the set
	${\rm U}(f,g_1,g_2,\ldots,g_m)$, which play a key role in our study of the S-lemma
	and strong duality for quadratic programming.
	\begin{theorem}\label{Theo3} 
		Suppose that there exist $i_0\in I$ and $\lambda\in \mathbb{R}$ such that $B_{i_0}\succ 0$ and
		\begin{description}
			\item[${\rm (H_1)}$] $A+\lambda B_{i_0}\succeq 0$;
			\item[${\rm (H_2)}$] $[f+\lambda g_{i_0}\leq 0,\ g_i\leq 0,\ i\in I_0]^\infty \neq \{0\}$,
			where $I_0:=I\setminus \{i_0\}$.
		\end{description}	 
		Then the set
		$$
		{\rm U}(f,g_1,\ldots,g_m)
		=\{(f(x), g_1(x), \ldots, g_m(x)) : x\in \mathbb{R}^n\}+\mathbb{R}^{m+1}_+
		$$
		is convex.                    
	\end{theorem}
	
	\begin{proof}
		Let
		$$
		F(x):=f(x)+\lambda g_{i_0}(x).
		$$
		For each $\omega\in \mathbb{R}^m$, consider the quadratic programming problem
		\begin{equation*}\tag{$GP_\omega$}
			\min \{F(x): x\in \mathbb{R}^n,\ g_i(x)\leq \omega_i,\ i\in I\}.
		\end{equation*}
		
		First, we show that there exists a solution $\bar x$ of the problem $(GP_\omega)$ such that
		$g_{i_0}(\bar x)=\omega_{i_0}$.
		Indeed, let $\bar x$ be any solution of $(GP_\omega)$.
		If $g_{i_0}(\bar x)<\omega_{i_0}$, then by assumption ${\rm (H_2)}$ there exists
		$$
		0\neq \bar v\in [f+\lambda g_{i_0}\leq 0,\ g_i\leq 0,\ i\in I_0]^\infty.
		$$
		That is,
		$$
		(A+\lambda B_{i_0})\bar v=0,\quad
		(a+\lambda b_{i_0})^T\bar v\leq 0,\quad
		B_i\bar v=0,\quad
		b_i^T\bar v\leq 0 \quad \forall i\in I_0.
		$$
		Define $x(t):=\bar x+t\bar v$ for $t>0$.
		Since $B_{i_0}\succ 0$, we have $\bar v^T B_{i_0}\bar v>0$.
		Together with the inequality $g_{i_0}(\bar x)<\omega_{i_0}$, this implies that there exists
		$t_0>0$ such that
		\begin{equation}\label{c1}
			g_{i_0}(x(t_0))
			=g_{i_0}(\bar x)
			+t_0^2\bar v^TB_{i_0}\bar v
			+t_0(2B_{i_0}\bar x+b_{i_0})^T\bar v
			=\omega_{i_0}.
		\end{equation}
		Moreover, since $\bar v\in [g_i\leq 0,\ i\in I_0]^\infty$, we have
		$B_i\bar v=0$ and $b_i^T\bar v\leq 0$ for all $i\in I_0$.
		Hence, for each $i\in I_0$,
		\begin{equation}\label{c2}
			g_i(x(t_0))
			=g_i(\bar x)
			+t_0^2\bar v^TB_i\bar v
			+t_0(2B_i\bar x+b_i)^T\bar v
			\leq g_i(\bar x)
			\leq \omega_i.
		\end{equation}
		Combining \eqref{c1} and \eqref{c2}, we obtain
		\begin{equation}\label{c3}
			x(t_0)\in \{x\in \mathbb{R}^n : g_i(x)\leq \omega_i,\ i\in I\}.
		\end{equation}
		Furthermore, by the fact that $\bar v\in [f+\lambda g_{i_0}\leq 0]^\infty$, we have
		$(A+\lambda B_{i_0})\bar v=0$ and $(a+\lambda b_{i_0})^T\bar v\leq 0$.
		Consequently,
		\begin{align*}
			F(x(t_0))
			&=F(\bar x)
			+t_0^2\bar v^T(A+\lambda B_{i_0})\bar v
			+t_0\big(2(A+\lambda B_{i_0})\bar x+a+\lambda b_{i_0}\big)^T\bar v \\
			&\leq F(\bar x).
		\end{align*}
		Together with \eqref{c3}, this shows that $x(t_0)$ is also a solution of $(GP_\omega)$
		satisfying $g_{i_0}(x(t_0))=\omega_{i_0}$ and $g_i(x(t_0))\leq \omega_i$ for all $i\in I_0$.
		Thus, we may assume without loss of generality that there exists a solution $\bar x$ of
		$(GP_\omega)$ such that $g_{i_0}(\bar x)=\omega_{i_0}$.
		
		Next, consider the problem
		\begin{equation}\tag{$P_\omega$}
			\inf\{f(x): g_i(x)\leq \omega_i,\ i\in I\}.
		\end{equation}
		Let $\varphi(\omega)$ denote the optimal value of $(P_\omega)$.
		Define the feasible set
		$$
		\Phi(\omega):=\{x\in \mathbb{R}^n : g_i(x)\leq \omega_i,\ i\in I\}.
		$$
		If $\Phi(\omega)=\emptyset$, set $\varphi(\omega)=+\infty$; otherwise,
		$\varphi(\omega)=\inf\{f(x): x\in \Phi(\omega)\}$.
		Then
		$$
		{\rm U}(f,g_1,\ldots,g_m)={\rm epi}\,\varphi.
		$$
		
		We now show that $\varphi$ is convex.
		Indeed, we have
		\begin{align*}
			&\inf\{F(x): x\in \mathbb{R}^n,\ g_i(x)\leq \omega_i,\ i\in I\} \\
			=&\, f(\bar x)+\lambda g_{i_0}(\bar x) \\
			\geq&\, \inf\{f(x): x\in \mathbb{R}^n,\ g_i(x)\leq \omega_i,\ i\in I\}
			+\lambda \omega_{i_0} \\
			=&\, \inf\{f(x)+\lambda \omega_{i_0}: x\in \mathbb{R}^n,\ g_i(x)\leq \omega_i,\ i\in I\} \\
			\geq&\, \inf\{f(x)+\lambda g_{i_0}(x): x\in \mathbb{R}^n,\ g_i(x)\leq \omega_i,\ i\in I\} \\
			=&\, \inf\{F(x): x\in \mathbb{R}^n,\ g_i(x)\leq \omega_i,\ i\in I\}.
		\end{align*}
		Hence,
		\begin{align}
			\inf\{f(x): x\in \mathbb{R}^n,\ g_i(x)\leq \omega_i,\ i\in I\}
			=\inf\{F(x): x\in \mathbb{R}^n,\ g_i(x)\leq \omega_i,\ i\in I\}
			-\lambda \omega_{i_0}.
			\label{e1}
		\end{align}
		
		Since assumption ${\rm (H_1)}$ implies that $F$ is convex,
		it follows from \cite[Lemma~4.24]{Ruszczynski2006} that the optimal value function of $(GP_\omega)$
		is convex.
		Combining this with \eqref{e1}, we conclude that $\varphi$ is convex.
		Therefore, ${\rm U}(f,g_1,\ldots,g_m)={\rm epi}\,\varphi$ is convex.
	\end{proof}

	As an application of Theorem~\ref{Theo3}, we consider the following trust-region problem
	with linear inequality constraints:
	\begin{equation}\tag{$TP$}
		\inf\{f(x):=x^TAx+a^Tx : h_1(x)\leq 0,\ h_i(x)\leq 0,\ i=2,\ldots,m\},
	\end{equation}
	where $h_1(x):=\|x-x_0\|^2-\alpha$, $h_i(x):=b_i^Tx-\beta_i$,
	$A$ is a symmetric matrix,
	$a,b_i,x_0\in \mathbb{R}^n$,
	$\alpha,\beta_i\in \mathbb{R}$ with $\alpha>0$, and $i=2,\ldots,m$.
	
	Let $\bar\lambda$ denote the smallest eigenvalue of $A$ and consider the following condition:
	\begin{description}
		\item[${\rm (H_3)}$]
		The system
		\[
		Av-\bar\lambda v=0,\quad
		(a+2\bar\lambda x_0)^Tv\leq 0,\quad
		b_i^Tv\leq 0,\ i=2,\ldots,m,
		\]
		has a nonzero solution.
	\end{description}
	Applying Theorem~\ref{Theo3} to the problem $(TP)$, we obtain the following results.
	\begin{proposition}\label{Theo4}
		Consider the problem $(TP)$.
		If assumption ${\rm (H_3)}$ holds, then conditions ${\rm (H_1)}$ and ${\rm (H_2)}$
		are satisfied, and the set ${\rm U}(f,h_1,\ldots,h_m)$ is convex.
	\end{proposition}
	
	\begin{proof}
		Assume that ${\rm (H_3)}$ holds.
		Let $I_n$ denote the $n\times n$ identity matrix.
		Then condition ${\rm (H_1)}$ is satisfied with $\lambda=-\bar\lambda$,
		$i_0=1$, and $B_1=I_n$.
		To verify condition ${\rm (H_2)}$, note that
		$$h_1(x)=x^TI_nx-2x_0^Tx+\|x_0\|^2-\alpha$$
		and
		$$f(x)-\bar\lambda h_1(x)
		=x^T(A-\bar\lambda I_n)x+(a+2\bar\lambda x_0)^Tx
		-\bar\lambda\|x_0\|^2+\bar\lambda\alpha.$$
		It follows that
		\begin{align*}
			&[f-\bar\lambda h_1\leq 0,\ h_i\leq 0,\ i=2,\ldots,m]^\infty \\
			=&\{v\in \mathbb{R}^n :
			Av-\bar\lambda v=0,\ (a+2\bar\lambda x_0)^Tv\leq 0,\
			b_i^Tv\leq 0,\ i=2,\ldots,m\}.
		\end{align*}
		By assumption ${\rm (H_3)}$, the above recession cone is nontrivial, that is,
		$$
		[f-\bar\lambda h_1\leq 0,\ h_i\leq 0,\ i=2,\ldots,m]^\infty \neq \{0\}.
		$$
		Therefore, condition ${\rm (H_2)}$ is satisfied.
		Applying Theorem~\ref{Theo3}, we conclude that the set
		${\rm U}(f,h_1,\ldots,h_m)$ is convex.
	\end{proof}

	\begin{proposition}\label{Theo6} 
		Consider the problem $(TP)$.
		If the following assumption holds:
		\begin{description}
			\item[${\rm (H_4)}$]
			the system
			\[
			Av-\bar\lambda v=0,\quad b_i^Tv=0,\ i=2,\ldots,m
			\]
			has a nonzero solution,
		\end{description}
		then assumption ${\rm (H_3)}$ is satisfied and the set
		${\rm U}(f,h_1,\ldots,h_m)$ is convex.        
	\end{proposition}
	
	\begin{proof}
		Assume that ${\rm (H_4)}$ holds.
		Then there exists a nonzero vector $\bar v$ such that
		$A\bar v-\bar\lambda \bar v=0$ and $b_i^T\bar v=0$ for all $i=2,\ldots,m$.
		If $(a+2\bar\lambda x_0)^T\bar v\leq 0$, then assumption ${\rm (H_3)}$ is satisfied.
		Otherwise, let $v^0:=-\bar v$.
		Then
		$Av^0-\bar\lambda v^0=0$,
		$b_i^Tv^0=0$ for all $i=2,\ldots,m$,
		and $(a+2\bar\lambda x_0)^Tv^0<0$.
		Hence, assumption ${\rm (H_3)}$ holds in this case as well.
		The conclusion now follows directly from Proposition~\ref{Theo4}.
	\end{proof}

	\begin{remark}
		In \cite[Theorem~2.1]{JeyakumarLi2014}, the authors investigated the convexity of the set
		${\rm U}(f,h_1,\ldots,h_m)$ under the following dimension condition
		\begin{equation}\label{dksc}
			\dim \ker(A-\bar\lambda I_n)\geq s+1,
		\end{equation}
		where $s:=\dim \operatorname{span}\{b_2,\ldots,b_m\}$.
		We show that the dimension condition \eqref{dksc} is strictly stronger than
		assumption ${\rm (H_4)}$ in Proposition~\ref{Theo6}.
		More precisely, condition \eqref{dksc} implies ${\rm (H_4)}$.
		Indeed, suppose that ${\rm (H_4)}$ does not hold, that is,
		\[
		\ker(A-\bar\lambda I_n)\cap\Big(\bigcap_{i=2}^m b_i^\perp\Big)=\{0\}.
		\]
		Then, by the dimension formula for subspaces, we obtain
		\begin{align*}
			\dim \ker(A-\bar\lambda I_n)
			&=\dim\Big(\ker(A-\bar\lambda I_n)+\bigcap_{i=2}^m b_i^\perp\Big)\\
			&\quad+\dim\Big(\ker(A-\bar\lambda I_n)\cap\bigcap_{i=2}^m b_i^\perp\Big)
			-\dim\Big(\bigcap_{i=2}^m b_i^\perp\Big)\\
			&\leq n+0-(n-s)\\
			&=s.
		\end{align*}
		Consequently, condition \eqref{dksc} fails.
		This shows that \eqref{dksc} is a strictly stronger condition than ${\rm (H_4)}$.
	\end{remark}

	The following examples not only illustrate our obtained results but also demonstrate that
	the existing results in the literature cannot be applied to these problems.
	
	\begin{example}\label{ex1}
		Consider  problem $(TP)$ with $n=1$ and $m=2$, where
		$f(x)=-x^2+x$, $h_1(x)=x^2-1$, and $h_2(x)=x$ for $x\in\mathbb{R}$.
		Then $A=-1$, $a=1$, $x_0=0$, $\bar\lambda=-1$, and $b_2=1$.
		We have
		\[
		\{v\in\mathbb{R}: Av-\bar\lambda v=0,\ (a+2\bar\lambda x_0)^Tv\le0,\ b_2^Tv\le0\}
		=\{v\in\mathbb{R}: v\le0\}\neq\{0\}.
		\]
		Hence, assumption ${\rm (H_3)}$ is satisfied, and consequently ${\rm (H_2)}$ holds.
		Therefore, by Theorem~\ref{Theo3} and Proposition~\ref{Theo4}, the set
		${\rm U}(f,h_1,h_2)$ is convex.
		
		On the other hand, we have
		\[
		\dim \ker(A-\bar\lambda I_1)=\dim \mathbb{R}=1
		\quad\text{and}\quad
		s:=\dim \operatorname{span}\{b_2\}=1.
		\]
		Thus,
		\[
		\dim \ker(A-\bar\lambda I_1)< s+1.
		\]
		This shows that the dimension condition \eqref{dksc} is not satisfied.
		Consequently, Theorem~2.1 in \cite{JeyakumarLi2014} cannot be applied to this problem.
	\end{example}
	
	\begin{example}\label{ex2}
		Consider problem $(TP)$ with $n=2$ and $m=2$, where
		$f(x)=-x_1^2+x_2$, $h_1(x)=x_1^2+x_2^2-1$, and $h_2(x)=x_2$ for
		$x=(x_1,x_2)\in\mathbb{R}^2$.
		We have
		\[
		A=\begin{pmatrix}
			-1 & 0\\
			0 & 0
		\end{pmatrix},\quad
		a=\begin{pmatrix}
			0\\
			1
		\end{pmatrix},\quad
		b_2=\begin{pmatrix}
			0\\
			1
		\end{pmatrix},\quad
		x_0=0,\quad
		\bar\lambda=-1.
		\]
		Then
		\[
		\{v\in\mathbb{R}^2: Av-\bar\lambda v=0,\ b_2^Tv=0\}
		=\{(v_1,v_2)\in\mathbb{R}^2: v_2=0\}\neq\{0\}.
		\]
		Hence, assumption ${\rm (H_4)}$ is satisfied.
		By Proposition~\ref{Theo6}, the set ${\rm U}(f,h_1,h_2)$ is convex.
		
		However, we have
		\[
		\dim \ker(A-\bar\lambda I_2)
		=\dim\{(v_1,v_2)\in\mathbb{R}^2: v_2=0\}=1,
		\quad
		s:=\dim \operatorname{span}\{b_2\}=1,
		\]
		and therefore
		\[
		\dim \ker(A-\bar\lambda I_2)< s+1.
		\]
		Thus, the dimension condition \eqref{dksc} fails, and Theorem~2.1 in
		\cite{JeyakumarLi2014} cannot be applied to the problem in this example.
	\end{example}
	
	The following example shows that Theorem~\ref{Theo3} and
	Propositions~\ref{Theo4} and~\ref{Theo6} may fail if any one of the
	assumptions ${\rm (H_2)}$--${\rm (H_4)}$ is omitted.
	
	\begin{example}\label{ex3}
		Consider  problem $(TP)$ with $n=1$ and $m=2$, where
		$f(x)=-x^2+x$, $h_1(x)=x^2-1$, and $h_2(x)=-x$ for $x\in\mathbb{R}$.
		We have $A=-1$, $a=1$, $x_0=0$, $\bar\lambda=-1$, and $b_2=-1$.
		It can be verified that the set ${\rm U}(f,h_1,h_2)$ is not convex
		(see \cite[Example~2.1]{JeyakumarLi2014}).
		
		We observe that
		\[
		\{v\in\mathbb{R}: Av-\bar\lambda v=0,\ (a+2\bar\lambda x_0)^Tv\le0,\ b_2^Tv\le0\}
		=\{0\}.
		\]
		Hence, assumption ${\rm (H_3)}$ is not satisfied.
		
		Moreover, we have
		\[
		\{v\in\mathbb{R}: Av-\bar\lambda v=0,\ b_2^Tv=0\}
		=\{0\},
		\]
		which shows that condition ${\rm (H_4)}$ also fails.
		
		Finally, for any $\lambda\in \mathbb{R}^n$ satisfying $A+\lambda I_1\succeq 0$, we have
		\[
		[f-\lambda h_1\le0,\ h_2\le0]^\infty
		=\{v\in\mathbb{R}: v=0\}.
		\]
		Therefore, assumption ${\rm (H_2)}$ is not satisfied either.
	\end{example}
	
	Consider the following CDT (Celis--Dennis--Tapia, or two-ball trust-region) problem
	\begin{equation}\tag{{\rm $AP$}}
		\inf\{x^T(A-\bar \lambda I_n)x+a^Tx: \|x\|^2\leq 1,\ \|Cx-c\|^2\leq 1,\ Dx-d\leq 0\},
	\end{equation}
	where $A$ is a symmetric matrix, $C$ is an $(l\times n)$ matrix, $D$ is an $(m\times n)$ matrix, $c\in \mathbb{R}^l$, and $d\in \mathbb{R}^m$.
	
	Applying Theorem~\ref{Theo3} to problem {\rm $(AP)$} yields the following corollary.
	
	\begin{corollary}\label{cor1}
		Consider the problem {\rm $(AP)$}. If the following condition holds
		\begin{equation}\label{dkb}
			{\rm Ker}(A-\bar \lambda I_n)\cap {\rm Ker}(C)\cap \{v\in \mathbb{R}^n: Dv\leq 0,\ a^Tv\leq 0\}\neq\{0\},
		\end{equation}
		then problem {\rm $(AP)$} admits a minimizer on the sphere
		$\{x\in \mathbb{R}^n:\|x\|=1\}$.
	\end{corollary}
	
	\begin{proof}
		Assume that condition \eqref{dkb} is satisfied. It is straightforward to verify that assumptions ${\rm (H_1)}$ and ${\rm (H_2)}$ hold. By repeating the arguments used in the proof of Theorem~\ref{Theo3}, we obtain the desired conclusion.
	\end{proof}
	
	\begin{remark}
		The result stated in Corollary~\ref{cor1} was first established in
		\cite[Theorem~5.1 and Lemma~5.1]{BomzeJeyakumarLi2018}. Condition~\eqref{dkb} presented here corrects condition~$(15)$ in \cite[Theorem~5.1]{BomzeJeyakumarLi2018}.
	\end{remark}
	
	\begin{remark}
		In \cite[Example~3]{LaraUrruty2022}, the authors conjectured that the convexity of the set ${\rm U}(f,g_1,\ldots,g_m)$ does not hold for $m\geq 2$, although no counterexample was provided. Through Examples~\ref{ex1}--\ref{ex3}, we confirm that this conjecture is indeed true.
	\end{remark}
	
	We now obtain the following corollary concerning the hidden convexity of the classical trust-region system.
	
	\begin{corollary}{\rm (see \cite[Theorem~2.2]{Polyak1998})}
		Let $f(x):=x^TAx+a^Tx+\beta$ and let $h_1(x):=\|x-x_0\|^2-\alpha$, where
		$A\in \mathbb{R}^{n\times n}_{+}$, $a,x_0\in \mathbb{R}^n$, and
		$\alpha,\beta\in \mathbb{R}$. Then the set ${\rm U}(f,h_1)$ is convex.
	\end{corollary}
	
	\begin{proof}
		The conclusion follows directly from Theorem~\ref{Theo3} and Proposition~\ref{Theo6}.
	\end{proof}

	\section{Hidden convexity of a system of two quadratic inequalities}
	In this section, we study the convexity of the following set
	$$
	C:= \{(f(x), g(x)):\; x\in M\}+\mathbb{R}^{2}_+,
	$$
	where $M$ is an affine manifold in $\mathbb{R}^n$ and $f,g$ are quadratic functions on $\mathbb{R}^n$.
	
	The following theorem, proposed in \cite[Corollary~10]{TuyTuan2013}, is an important result concerning the hidden convexity of systems of quadratic inequalities.
	
	\begin{theorem}{\rm(\cite[Corollary~10]{TuyTuan2013}, \cite[Corollary~10.10]{Tuy2016})}\label{Theo3m} 
		Suppose that $M$ is an affine manifold in $\mathbb{R}^n$ and $f,g$ are quadratic functions on $\mathbb{R}^n$.	 
		Then the set
		$$
		C= \{(f(x), g(x)):\; x\in M\}+\mathbb{R}^{2}_+
		$$
		is convex.                    
	\end{theorem}
	
	The proof of Theorem~\ref{Theo3m} given in \cite{TuyTuan2013} contains a gap, since the set $C$ is not closed in general. This issue was addressed in \cite[Corollary~10.10]{Tuy2016}. However, the revised proof still contains an illogical argument. More precisely, Corollary~10.8 in \cite{Tuy2016} cannot be applied to the case where there exists no $x\in [a,b]$ such that $f(x)<t_1'$ and $g(x)<t_2'$. Recently, Theorem~\ref{Theo3m} was proved in \cite{BazanOpazo2016} using a different approach. In this section, we provide a complete  proof of Theorem~\ref{Theo3m}. To this end, we first establish the following proposition.
	
	\begin{proposition}\label{Theo3n}
		Suppose that $\alpha(x):=a_1x^2+a_2x+a_3$ and $\beta(x):=b_1x^2+b_2x+b_3$, where $a_i,b_i\in\mathbb{R}$ for $i=1,2,3$, and $x\in\mathbb{R}$. 	 
		Then the set
		$$
		V:= \{(\alpha(x), \beta(x)):\; x\in \mathbb{R}\}+\mathbb{R}^{2}_+
		$$
		is convex.                   
	\end{proposition}
	
	\begin{proof}
		If $a_1\geq 0$ and $b_1\geq 0$, then both $\alpha$ and $\beta$ are convex functions. Consequently, the set $\{(\alpha(x), \beta(x)):\; x\in \mathbb{R}\}$ is convex, and hence $V$ is convex. We now consider the remaining cases where at least one of $\alpha$ and $\beta$ is not convex.
		
		\medskip
		\noindent
		{\it Case 1: $a_1>0$ or $b_1>0$.}
		Without loss of generality, assume that $a_1>0$. Let $\{t^k\}\subset V$ be a sequence converging to $\bar t=(\bar t_1,\bar t_2)\in \mathbb{R}^2$. For each $k$, since $t^k=(t^k_1,t^k_2)\in V$, there exists $z^k\in \mathbb{R}$ such that
		$\alpha(z^k)\leq t^k_1$ and $\beta(z^k)\leq t^k_2$.
		Since $\alpha$ is strictly convex, each sublevel set $\{x\in \mathbb{R}:\alpha(x)\leq t\}$ is compact. Hence, the sequence $\{z^k\}$ is bounded.
		
		Passing to a subsequence if necessary, we may assume that $z^k\to \bar z\in \mathbb{R}$ as $k\to\infty$. By the continuity of $\alpha$ and $\beta$, we obtain
		$\alpha(\bar z)\leq \bar t_1$ and $\beta(\bar z)\leq \bar t_2$.
		Thus, $V$ is closed.
		
		Next, we prove the convexity of $V$. Let $t=(t_1,t_2)\in \mathbb{R}^2\setminus V$. Since $V$ is closed, there exists $\tilde t=(\tilde t_1,\tilde t_2)\notin V$ such that $\tilde t_1>t_1$ and $\tilde t_2>t_2$. Consequently, there exists no $x\in \mathbb{R}$ satisfying
		$$
		\alpha(x)\leq t_1<\tilde t_1
		\quad \text{and} \quad
		\beta(x)\leq t_2<\tilde t_2.
		$$
		By \cite[Corollary~8]{TuyTuan2013}, there exists $\lambda=(\lambda_1,\lambda_2)\in \mathbb{R}^2_+\setminus\{0\}$ such that
		$$
		\lambda_1(\alpha(x)-\tilde t_1)+\lambda_2(\beta(x)-\tilde t_2)\geq 0
		\quad \forall x\in \mathbb{R}.
		$$
		Define
		$$
		M_t:=\{u\in \mathbb{R}^2:\lambda^Tu-\lambda^T\tilde t\geq 0\}.
		$$
		Since $\tilde t_1>t_1$ and $\tilde t_2>t_2$, we have $\lambda^Tt-\lambda^T\tilde t<0$, and thus $t\notin M_t$. For any $v\in V$, one has $\tilde t\leq v$, which implies $\lambda^Tv-\lambda^T\tilde t\geq 0$. Hence, $V\subset M_t$.
		
		Therefore, for each $t\notin V$, there exists a half-space $M_t$ separating $t$ from $V$. This shows that
		$$
		V=\bigcap_{t\notin V} M_t
		$$
		is convex.
		
		\medskip
		\noindent
		{\it Case 2: $a_1=0$, $a_2<0$, and $b_1=0$.}
		Then $\alpha(x)=c_1$ for all $x\in \mathbb{R}$ and $\beta(x)\to -\infty$ as $x\to -\infty$. Hence,
		$V=[c_1,+\infty)\times\mathbb{R}$, which is convex.
		
		\medskip
		\noindent
		{\it Case 3: $a_1=0$, $a_2<0$, and $b_1\neq 0$.}
		
		\noindent
		{\it Case 4: $a_1<0$ and $a_2<0$.}
		
		\noindent
		{\it Case 5: $a_1<0$, $a_2=0$, and $b_2\neq 0$.}
		
		In Cases~3--5, we have $(\alpha(x),\beta(x))\to (-\infty,-\infty)$ as $x\to \pm\infty$. Thus, $V=\mathbb{R}^2$, which is convex.
		
		\medskip
		\noindent
		{\it Case 6: $a_1<0$, $a_2=0$, and $b_2=0$.}
		Then $\alpha(x)\to -\infty$ as $x\to -\infty$ and $\beta(x)=c_2$ for all $x\in \mathbb{R}$. Hence,
		$V=\mathbb{R}\times[c_2,+\infty)$, which is convex.
		
		The proof is complete.
	\end{proof}
	
	\medskip
	\noindent
	\textbf{Proof of Theorem~\ref{Theo3m}.}
	
	For any two points $a,b\in M$, let $M_{ab}$ denote the affine manifold spanned by $a$ and $b$, and define
	$$
	C_{ab}:= \{(f(x), g(x)):\; x\in M_{ab}\}+\mathbb{R}^{2}_+.
	$$
	For each $x\in M_{ab}$, there exists $t\in \mathbb{R}$ such that $x=ta+(1-t)b$. Hence,
	$$
	C_{ab}= \{(\tilde f(t), \tilde g(t)):\; t\in \mathbb{R}\}+\mathbb{R}^{2}_+,
	$$
	where $\tilde f$ and $\tilde g$ are quadratic functions of $t$ defined by
	$$
	\tilde f(t)=f(ta+(1-t)b), \qquad
	\tilde g(t)=g(ta+(1-t)b).
	$$
	By Proposition~\ref{Theo3n}, the set $C_{ab}$ is convex.
	
	To prove the convexity of $C$, let $u=(u_1,u_2)$ and $v=(v_1,v_2)$ be arbitrary points in $C$. Then there exist $\tilde a,\tilde b\in M$ such that
	$$
	f(\tilde a)\leq u_1,\quad g(\tilde a)\leq u_2,
	$$
	and
	$$
	f(\tilde b)\leq v_1,\quad g(\tilde b)\leq v_2.
	$$
	This implies that $u,v\in C_{\tilde a\tilde b}$. Since $C_{\tilde a\tilde b}$ is convex, we have
	$$
	[u,v]\subset C_{\tilde a\tilde b}\subset C.
	$$
	Therefore, $C$ is convex. The proof is complete.

	\section{S-lemma, global optimality and strong duality}
	
	In this section, we first provide necessary and sufficient conditions for the S-lemma associated with systems of quadratic inequalities. We also establish strong duality results for quadratic programming.
	
	\begin{theorem}\label{Theo7} 
		Consider the following statements:
		\begin{description}
			\item{$(i)$} {\rm [S-lemma]} For each quadratic function $f:\mathbb{R}^n\to \mathbb{R}$,
			\begin{align*}
				&[g_i(x)\leq 0,\; i=1, \ldots, m \Rightarrow f(x)\geq 0] \\[0.3em]
				\Leftrightarrow \; \; &(\exists \lambda_i\geq0,\; i=1, \ldots, m)\; (\forall x\in \mathbb{R}^n)\;  
				f(x)+\sum_{i=1}^{m}\lambda_ig_i(x)\geq 0;
			\end{align*}
			
			\item{$(ii)$} The set
			$$\bigcup_{\lambda\in  \mathbb{R}^m_+} \; {\rm epi}\bigg(\sum_{i=1}^{m}\lambda_ig_i\bigg)^*$$
			is closed;
			
			\item{$(iii)$} {\rm [Slater's condition]} There exists $x^0\in \mathbb{R}^n$ such that
			$g_i(x^0)< 0,\; i=1, \ldots, m$;
			
			\item{$(iv)$} {\rm [Strong duality]} For each quadratic function $f$,
			\begin{equation}\label{e11} 
				\inf \{f(x): g_i(x)\leq 0,\; i=1,\ldots,m\}
				=\max_{\lambda\in \mathbb{R}^m_+}\inf_{x\in \mathbb{R}^n}
				\bigg\{f(x)+\sum_{i=1}^{m}\lambda_ig_i(x)\bigg\}.
			\end{equation}
		\end{description}
		The following statements hold:
		\begin{description}
			\item{$(a)$} $(i)$ implies $(ii)$;
			
			\item{$(b)$} If ${\rm U}(f, g_1, g_2,\ldots,g_m)$ is a convex set, then $(iii)$ implies $(i)$; 
			
			\item{$(c)$} If the assumptions $({\rm H_1})$ and $({\rm H_2})$ hold, then $(iii)$ implies $(i)$;
			
			\item{$(d)$} $(i)$ and $(iv)$ are equivalent;
			
			\item{$(e)$} If $m=1$, then the statements $(i)$--$(iv)$ are equivalent.
		\end{description}
	\end{theorem}
	
	\begin{proof}
		$(a)$ Suppose that $(i)$ holds. Let
		\[
		(u,\alpha)\in {\rm cl}\bigcup_{\lambda\in \mathbb{R}^m_+} 
		{\rm epi}\bigg(\sum_{i=1}^{m}\lambda_ig_i \bigg)^*.
		\]
		Then there exist sequences $\{u^k\}\subset\mathbb{R}^n$, 
		$\{\alpha^k\}\subset\mathbb{R}$, and $\{\lambda^k\}\subset \mathbb{R}^m_+$ such that, for each $k$,
		\[
		(u^k,\alpha^k)\in {\rm epi}\bigg(\sum_{i=1}^{m}\lambda_i^kg_i \bigg)^*,
		\]
		with $\lim_{k\to \infty}u^k=u$ and $\lim_{k\to \infty}\alpha^k=\alpha$.
		Hence,
		\[
		\bigg(\sum_{i=1}^{m}\lambda_i^kg_i \bigg)^*(u^k)\leq \alpha^k.
		\]
		That is, for each $x\in \mathbb{R}^n$,
		\[
		(u^k)^Tx-\sum_{i=1}^{m}\lambda_i^kg_i(x) \leq \alpha^k.
		\]
		
		For each $x\in \mathbb{R}^n$ satisfying $g_i(x)\leq 0$, $i=1,\ldots,m$, we have
		\[
		\sum_{i=1}^{m}\lambda_i^kg_i(x)\leq 0,
		\]
		and hence $(u^k)^Tx\leq \alpha^k$.
		Letting $k\to \infty$, we obtain
		\[
		-u^Tx+\alpha\geq 0.
		\]
		Define
		\[
		f(x):=-u^Tx+\alpha.
		\]
		Then $f$ is a quadratic function. By assumption $(i)$, there exist $\mu_i\geq 0$, $i=1,\ldots,m$, such that
		\[
		-u^Tx+\alpha+\sum_{i=1}^{m}\mu_ig_i(x)\geq 0
		\quad \text{for all } x\in \mathbb{R}^n.
		\]
		This implies
		\[
		u^Tx-\sum_{i=1}^{m}\mu_ig_i(x)\leq \alpha
		\quad \text{for all } x\in \mathbb{R}^n,
		\]
		and therefore
		\[
		\bigg(\sum_{i=1}^{m}\mu_ig_i\bigg)^*(u)\leq\alpha.
		\]
		Consequently,
		\[
		(u,\alpha)\in {\rm epi}\bigg(\sum_{i=1}^{m}\mu_ig_i\bigg)^*
		\subset \bigcup_{\lambda\in \mathbb{R}^m_+} 
		{\rm epi}\bigg(\sum_{i=1}^{m}\lambda_ig_i \bigg)^*.
		\]
		Thus, $\bigcup_{\lambda\in \mathbb{R}^m_+} 
		{\rm epi}\big(\sum_{i=1}^{m}\lambda_ig_i \big)^*$ is closed, and $(ii)$ follows.
		
		\medskip
		$(b)$ By \cite[Theorem~1]{LaraUrruty2022}, the desired conclusion follows immediately.
		
		\medskip
		$(c)$ Under assumptions $({\rm H_1})$ and $({\rm H_2})$, Theorem~\ref{Theo3} ensures that the set
		${\rm U}(f, g_1, g_2,\ldots,g_m)$ is convex. Hence, the conclusion follows from part $(b)$.
		
		$(d)$ We first show that $(i)\Rightarrow(iv)$. 
		Suppose that $(i)$ holds. For each quadratic function $f$, the weak duality inequality always holds:
		\[
		\inf \{f(x): g_i(x)\leq 0,\ i=1,\ldots,m\}
		\geq
		\max_{\lambda\in  \mathbb{R}^m_+}
		\inf_{x\in \mathbb{R}^n}
		\bigg\{f(x)+\sum_{i=1}^{m}\lambda_ig_i(x)\bigg\}.
		\]
		Let
		\[
		f^*:=\inf \{f(x): g_i(x)\leq 0,\ i=1,\ldots,m\}.
		\]
		We consider the following two cases.
		
		\medskip
		\noindent
		{\it Case 1:} $f^*=-\infty$.  
		In this case, $(iv)$ holds trivially.
		
		\medskip
		\noindent
		{\it Case 2:} $f^*>-\infty$.  
		Then $f(x)-f^*\geq 0$ for every $x\in \{x\in \mathbb{R}^n: g_i(x)\leq 0,\ i=1,\ldots,m\}$. 
		By assumption $(i)$, there exist $\bar\lambda_i\geq 0$, $i=1,\ldots,m$, such that
		\[
		f(x)+\sum_{i=1}^{m}\bar\lambda_ig_i(x)\geq f^*,
		\quad \forall x\in \mathbb{R}^n.
		\]
		Consequently,
		\[
		\max_{\lambda\in  \mathbb{R}^m_+}
		\inf_{x\in \mathbb{R}^n}
		\bigg\{f(x)+\sum_{i=1}^{m}\lambda_ig_i(x)\bigg\}
		\geq
		\inf_{x\in \mathbb{R}^n}
		\bigg\{f(x)+\sum_{i=1}^{m}\bar\lambda_ig_i(x)\bigg\}
		\geq f^*.
		\]
		Hence, $(iv)$ follows.
		
		\medskip
		We next prove that $(iv)\Rightarrow(i)$. 
		Suppose, by contradiction, that $(i)$ does not hold. Then there exist a quadratic function
		$\bar f:\mathbb{R}^n\to \mathbb{R}$ and a point $\bar x\in \mathbb{R}^n$ such that
		$\bar f(x)\geq 0$ for every $x\in [g\leq 0]$, but
		\[
		\bar f(\bar x)+\sum_{i=1}^{m}\lambda_ig_i(\bar x)<0
		\quad \text{for all } \lambda=(\lambda_1,\ldots,\lambda_m)\in \mathbb{R}^m_+.
		\]
		This implies that
		\[
		\inf_{x\in \mathbb{R}^n}
		\bigg\{\bar f(x)+\sum_{i=1}^{m}\lambda_ig_i(x)\bigg\}<0,
		\quad \forall \lambda\in \mathbb{R}^m_+,
		\]
		and hence
		\[
		\max_{\lambda\in  \mathbb{R}^m_+}
		\inf_{x\in \mathbb{R}^n}
		\bigg\{\bar f(x)+\sum_{i=1}^{m}\lambda_ig_i(x)\bigg\}<0.
		\]
		On the other hand, since $\bar f(x)\geq 0$ for every $x\in [g\leq 0]$, we have
		\[
		\inf \bigg\{\bar f(x): g_i(x)\leq 0,\ i=1,\ldots,m\bigg\}\geq 0.
		\]
		Therefore,
		\[
		\inf \bigg\{\bar f(x): g_i(x)\leq 0,\ i=1,\ldots,m\bigg\}
		>
		\max_{\lambda\in  \mathbb{R}^m_+}
		\inf_{x\in \mathbb{R}^n}
		\bigg\{\bar f(x)+\sum_{i=1}^{m}\lambda_ig_i(x)\bigg\},
		\]
		which contradicts $(iv)$. Hence, $(i)$ must hold.
		
		\medskip
		$(e)$ According to parts $(a)$ and $(b)$, it suffices to show that $(ii)$ implies $(iii)$. 
		This implication follows directly from \cite[Theorem~3.1]{JeyakumarHuyLi2009}.
		The proof is complete. 
	\end{proof}
	
	\begin{remark}
		According to Theorem~\ref{Theo7} (parts $(c)$ and $(d)$), the strong duality result \eqref{e11} holds under the assumptions that $({\rm H_1})$, $({\rm H_2})$, and Slater’s condition are satisfied. 
		
		For an affine manifold $H\subset \mathbb{R}^n$ and quadratic functions $f$ and $g$, the authors of \cite[Theorem~4]{TuyTuan2013} established the strong duality result \eqref{e11} for the problem
		\begin{equation}\label{e16} 
			\inf \{f(x): x\in H,\ g_1(x)\leq 0,\ g_2(x)\leq 0\}
		\end{equation}
		under the following two conditions:
		\begin{description}
			\item[$(i^\prime)$] There exists $x^*\in H$ such that $g_2(x^*)<0$;
			\item[$(ii^\prime)$] The function
			\[
			\omega(t):=\inf \{f(x)+tg_1(x): x\in H,\ g_2(x)\leq 0\}
			\]
			admits a maximizer $\bar t$ on $\mathbb{R}_+$, and one of the functions $f+\bar t g_1$ or $g_2$ is strictly convex on $H$, while the other is convex on $H$.
		\end{description}
		The following example illustrates an application of Theorem~\ref{Theo7}, whereas the result of \cite[Theorem~4]{TuyTuan2013} is not applicable to this problem.
	\end{remark}

	\begin{example}
		Consider the following problem $(TP)$:
		\begin{equation}\label{e16a} 
			\inf \{f(x)=2x_1x_2: x=(x_1,x_2)\in \mathbb{R}^2,\ x_1^2+x_2^2-1\leq0,\ -x_1\leq 0\}.
		\end{equation}
		Let $g_1(x)=x_1^2+x_2^2-1$ and $g_2(x)=-x_1$. Then,
		\[
		A=\begin{pmatrix} 
			0 & 1\\
			1 & 0 
		\end{pmatrix},\quad
		a=\begin{pmatrix} 
			0\\
			1
		\end{pmatrix},\quad
		b_2=\begin{pmatrix}
			-1\\
			0
		\end{pmatrix},\quad
		x_0=0,\quad
		\bar\lambda=-1.
		\]
		We have
		\begin{align*}
			&\{v\in \mathbb{R}^2: Av-\bar \lambda v=0,\ (a+2\bar \lambda x_0)^Tv\leq0,\ b_2^Tv\leq0\}\\
			=&\{v\in \mathbb{R}^2: v_1+v_2=0,\ v_1\geq 0\}\neq\{0\}.
		\end{align*}
		Hence, condition $({\rm H_3})$ is satisfied and, by Proposition~\ref{Theo4}, conditions $({\rm H_1})$ and $({\rm H_2})$ also hold.
		
		Moreover, Slater’s condition is satisfied since there exists
		$x^*=(\tfrac12,0)\in\mathbb{R}^2$
		such that
		$$
		g_1(x^*)=-\tfrac34<0
		\quad\text{and}\quad
		g_2(x^*)=-\tfrac12<0.
		$$
		According to Theorem~\ref{Theo7} (parts $(c)$ and $(d)$), the strong duality result \eqref{e11} holds, that is,
		\begin{equation*}
			\inf \{f(x): g_1(x)\leq 0,\ g_2(x)\leq 0\}
			=
			\max_{(\lambda_1,\lambda_2)\in \mathbb{R}^2_+}
			\inf_{x\in \mathbb{R}^2}
			\bigg\{f(x)+\lambda_1g_1(x)+\lambda_2g_2(x)\bigg\}.
		\end{equation*}
		
		However, for each $t\geq0$, we have
		\begin{align*}
			\omega(t)
			=&\inf \{f(x)+tg_1(x): g_2(x)\leq 0\}\\
			=&\inf \{2x_1x_2+t(x_1^2+x_2^2-1): x_1\geq 0\}\\
			=&\inf \{(x_1+x_2)^2+(t-1)(x_1^2+x_2^2)-t: x_1\geq0\}.
		\end{align*}
		Consequently,
		\[
		\omega(t)=
		\begin{cases}
			-t, & \text{if } t\geq 1,\\
			-\infty, & \text{if } 0\leq t<1.
		\end{cases}
		\]
		Thus, the function $\omega(t)$ admits a maximizer $\bar t=1$ on $\mathbb{R}_+$. Since both
		\[
		f(x)+\bar t g_1(x)=(x_1+x_2)^2-1
		\quad\text{and}\quad
		g_2(x)=-x_1
		\]
		are not strictly convex, the result in \cite[Theorem~4]{TuyTuan2013} cannot be applied to this problem.
	\end{example}

	As an application of Theorem~\ref{Theo7}, we establish a necessary and sufficient
	condition for global optimality in quadratic programming.
	
	\begin{proposition}\label{pro1}
		Consider the following quadratic programming problem:
		\begin{equation}\tag{$QP$}
			\inf\{f(x): g_i(x)\leq 0,\; i=1, \ldots, m\}.
		\end{equation}
		Suppose that assumptions $({\rm H_1})$ and $({\rm H_2})$ are satisfied, and that there exists
		$x^0\in \mathbb{R}^n$ such that $g_i(x^0)<0$ for every $i=1, \ldots, m$.
		Let $\bar x$ be a feasible point of problem $(QP)$.
		Then, $\bar x$ is a global minimizer of $(QP)$ if and only if there exists
		$\lambda=(\lambda_1,\ldots,\lambda_m)\in \mathbb{R}^m_+$ such that the following conditions hold:
		\begin{description}
			\item{$(i)$} $2(A+\sum_{i=1}^{m}\lambda_iB_i)\bar x+ (a+\sum_{i=1}^{m}\lambda_ib_i)=0$;
			\item{$(ii)$} $\lambda_ig_i(\bar x)=0, \; i=1, \ldots,m$; and,
			\item{$(iii)$} $A+\sum_{i=1}^{m}\lambda_iB_i\succeq 0.$
		\end{description}
	\end{proposition}
	\begin{proof}
		According to Theorem~\ref{Theo7}, $\bar x$ is a global minimizer of $(QP)$ if and only if there exists
		$\lambda=(\lambda_1,\ldots,\lambda_m)\in \mathbb{R}^m_+$ such that
		\begin{equation}\label{e12}
			f(x)+\sum_{i=1}^{m}\lambda_i g_i(x)\geq f(\bar x), \qquad \forall x\in \mathbb{R}^n.
		\end{equation}
		From this inequality and the fact that $\bar x$ is a feasible point of $(QP)$, it follows that
		\begin{equation}\label{e13}
			\sum_{i=1}^{m}\lambda_i g_i(\bar x)=0.
		\end{equation}
		Since $\lambda_i\geq 0$ and $g_i(\bar x)\leq 0$ for every $i=1,\ldots,m$, we obtain
		\begin{equation}\label{e14}
			\lambda_i g_i(\bar x)=0, \qquad i=1,\ldots,m.
		\end{equation}
		
		Define
		$$\varphi(x):=f(x)+\sum_{i=1}^{m}\lambda_i g_i(x)-f(\bar x).$$
		Since $\bar x$ is a global minimizer of the unconstrained problem
		$\inf\{\varphi(x): x\in \mathbb{R}^n\}$, it follows that
		$$\nabla \varphi(\bar x)=0
		\quad \text{and} \quad
		\nabla^2 \varphi(\bar x)\succeq 0,$$
		that is,
		$$2\bigg(A+\sum_{i=1}^{m}\lambda_i B_i\bigg)\bar x
		+\bigg(a+\sum_{i=1}^{m}\lambda_i b_i\bigg)=0$$
		and
		$$A+\sum_{i=1}^{m}\lambda_i B_i\succeq 0.$$
		
		Conversely, suppose that conditions $(i)$--$(iii)$ are satisfied.
		Then $\bar x$ is a global minimizer of the convex quadratic programming problem
		$\inf\{\varphi(x): x\in \mathbb{R}^n\}$.
		Consequently, for all $x\in \mathbb{R}^n$, we have
		\begin{align*}
			f(x)
			&\geq f(x)+\sum_{i=1}^{m}\lambda_i g_i(x) \\
			&\geq f(\bar x)+\sum_{i=1}^{m}\lambda_i g_i(\bar x)
			= f(\bar x),
		\end{align*}
		which shows that $\bar x$ is a global minimizer of $(QP)$.
	\end{proof}

	Applying Theorem~\ref{Theo7} and Proposition~\ref{pro1} to problem $(TP)$,
	we obtain the following result, which extends
	\cite[Theorem~4.1 and Corollary~4.1]{JeyakumarLi2014}.
	
	\begin{corollary}\label{cor2}
		Consider problem $(TP)$.
		Suppose that Slater's condition is satisfied.
		If at least one of the assumptions ${\rm (H_3)}$ or ${\rm (H_4)}$ holds, then the following statements are true:
		\begin{description}
			\item[$(i)$]
			A feasible point $\bar x$ of $(TP)$ is a global minimizer of $(TP)$ if and only if there exists
			$\lambda=(\lambda_1,\ldots,\lambda_m)\in \mathbb{R}^m_+$ such that
			$$2(A+\lambda_1 I_n)\bar x
			+\big(a-2\lambda_1 x_0+\sum_{i=2}^{m}\lambda_i b_i\big)=0,$$
			$$\lambda_1\big(\|\bar x-x_0\|^2-\alpha\big)=0,\qquad
			\lambda_i\big(b_i^T\bar x-\beta_i\big)=0,\; i=2,\ldots,m,$$
			and
			$$A+\lambda_1 I_n\succeq 0.$$
			
			\item[$(ii)$]
			Strong duality holds for $(TP)$, that is,
			\begin{align*}
				&\inf \Big\{x^TAx+a^Tx:
				\|x-x_0\|^2-\alpha\leq0,\;
				b_i^Tx-\beta_i\leq0,\ i=2,\ldots,m\Big\} \\
				=&\max_{\lambda\in \mathbb{R}^m_+}
				\inf_{x\in \mathbb{R}^n}
				\bigg\{
				x^TAx+a^Tx
				+\sum_{i=2}^{m}\lambda_i(b_i^Tx-\beta_i)
				\bigg\}.
			\end{align*}
		\end{description}
	\end{corollary}
	
	\begin{proof}
		Since at least one of the assumptions ${\rm (H_3)}$ or ${\rm (H_4)}$ holds,
		it follows from Propositions~\ref{Theo4} and~\ref{Theo6} that the set
		${\rm U}(f,h_1,\ldots,h_m)$ is convex.
		Under Slater's condition, applying Theorem~\ref{Theo7} together with
		Proposition~\ref{pro1} to problem $(TP)$ yields the desired conclusions.
	\end{proof}

	\section{Conclusions}
	In this paper, we establish the hidden convexity of quadratic systems under assumptions ${\rm (H_1)}$ and ${\rm (H_2)}$. The hidden convexity of the trust-region problem with linear inequality constraints $(TP)$ is further investigated under assumptions ${\rm (H_3)}$ and ${\rm (H_4)}$. The newly proposed assumptions are compared with existing ones, showing that they are strictly weaker. We provide a complete proof of the hidden convexity for systems involving two quadratic functions. Necessary and sufficient conditions for the S-lemma associated with systems of quadratic inequalities are also derived. Moreover, under these weaker assumptions, we establish necessary and sufficient global optimality conditions as well as strong duality results for quadratic programming problems.

	\section*{Acknowledgments}
	The research of the first author is funded by the Ministry of Education and Training of Vietnam [grant number B.2024-SP2-04].

\end{document}